\documentclass{amsart}

 \usepackage{amssymb,latexsym,amsmath,amsthm}

\usepackage[numbers,sort&compress]{natbib}

 \renewcommand{\div}{\mathop{\mathrm{div}}\nolimits}

\newtheorem*{thm*}{Theorem A}

\newtheorem{thm}{Theorem}

\newtheorem{dfn}{Definition}

\newtheorem{note}{Notation}

\newtheorem{lemma}{Lemma}

\newtheorem{cor}{Corollary}

\newtheorem{conj}{Conjecture}

\begin{document}

\def\RR{{\mathbb{R}}}

\title[nonlocal elliptic systems in two dimensions]{One-dimensional symmetry for integral systems in two dimensions}
\maketitle

\begin{center}
\author{Mostafa Fazly \footnote{The author is pleased to acknowledge the support of University of Alberta Start-up Grant RES0019810 and National Sciences and Engineering Research Council of Canada (NSERC) Discovery Grant RES0020463.} 
}
\\
{\it\small Department of Mathematical and Statistical Sciences, University of Alberta}\\
{\it\small Edmonton, Alberta, Canada T6G 2G1}\\
{\it\small e-mail: fazly@ualberta.ca}\vspace{1mm}
\end{center}

\begin{abstract}   The purpose of this brief paper is to prove De Giorgi type results for stable solutions of the following nonlocal system of integral equations in two dimensions 
$$   L(u_i)    =   H_i(u)  \quad  \text{in} \ \  \RR^2 , $$ 
where $u=(u_i)_{i=1}^m$ for $u_i: \RR^n\to \RR$, $H=(H_i)_{i=1}^m$ is a general nonlinearity.  The operator $L$ is given by  
$$L(u_i (x)):=  \int_{\RR^2} [u_i(x) - u_i(z)] K(z-x) dz,$$
for some kernel $K$.   The idea is to apply a linear Liouville theorem for the quotient of partial derivatives, just like in the proof of the classical De Giorgi's conjecture in lower dimensions.   Since there is no Caffarelli-Silvestre local extension problem associated to the above operator,   we deal with this problem directly via certain integral estimates.

\end{abstract}


\section{Introduction} 
We study the following semilinear elliptic system of nonlocal equations 
 \begin{equation} \label{main}
   L(u_i)    =   H_i(u)  \quad  \text{in} \ \  \RR^n 
  \end{equation}   
  where $u=(u_i)_{i=1}^m$ for $u_i: \RR^n\to \RR$ and $H=(H_i)_{i=1}^m$ is a sequence of locally Lipschitz functions.   The operator $L$ is an integral operator of convolution type 
  \begin{equation}\label{Lui}
L(u_i (x))=  \int_{\RR^n} [u_i(x) - u_i(z)] K(z-x) dz,  
\end{equation}
where the kernel $K$ satisfies a certain number of conditions.   We are interested in  De Giorgi type results for this system in two dimensions with a general nonlinearity $H$.    The case of scalar equations, that is when $m=1$,  is considered very recently by Ros-Oton and Sire in \cite{ros} and by Hamel and Valdinoci in \cite{hv}.     

Regarding semilinear equations, Ennio De Giorgi (1978) in \cite{De} proposed the following conjecture that has been a great inspiration for many authors in the field.   

\begin{conj} 
Suppose that $u$ is an entire solution of the Allen-Cahn equation that is 
\begin{equation}\label{allen}
\Delta u+u-u^3 =0	 \quad \text{on}\ \  \mathbb{R}^n, 
\end{equation}
satisfying $|u({x})| \le 1$, $\frac{\partial u}{\partial x_n} ({x}) > 0$ for ${x} = ({x}',x_n) \in \mathbb{R}^n$.	
Then, at least in dimensions $n\le 8$ the level sets of $u$ must be hyperplanes, i.e. there exists $g \in C^2(\mathbb{R})$ such that $u({x}) = g(a{x}' - x_n)$, for some fixed $a \in \mathbb{R}^{n-1}$.
\end{conj}
Ghoussoub and Gui in \cite{gg1} provided the first comprehensive affirmative answer on the De Giorgi conjecture in two dimensions.   In fact their proof if valid for a general Lipschitz nonlinearity and not necessarily double-well potentials.      Their proof uses the following, by now  standard, linear Liouville type theorem for elliptic equations in the divergence form, which (only) holds in dimensions one  and two, see \cite{bar, bbg, gg1}. If $\phi>0$, then any solution $\sigma$ of 
\begin{equation}
\div(\phi^2 \nabla \sigma)=0, 
\end{equation}
such that $\phi\sigma$ is bounded, is necessarily constant. This Liouville theorem is then applied to 
 the ratio $\sigma :=   \frac{\partial u}{\partial x_1}  / \frac{\partial u}{\partial x_2}  $ to finish the proof in two dimensions.  Ambrosio and Cabr\'{e} in  \cite{ac}, and later with Alberti in \cite{aac} for a general Lipschitz nonlinearity, provided a proof  for the conjecture  in three  dimensions by noting  that for the linear Liouville theorem to hold,  it suffices that 
\begin{equation}
\int_{B_{2R}\setminus B_R}\phi^2\sigma^2 \leq CR^2, 
\end{equation} 
for any $R>1$.   Then by showing that any solution $u$ such that $\partial_{x_n} u>0$ satisfies the energy estimate 
\begin{equation}
\int_{B_R}\phi^2\sigma^2 \leq CR^{n-1}, 
\end{equation} 
they finished the proof.    Even though the original conjecture remains open in dimensions $4\leq n \leq 8$,  there are various partial, yet groundbreaking, results for dimensions $4\leq n \leq 8$.  In this regard,  Ghoussoub and Gui showed in \cite{gg2} that the conjecture holds in  four and five dimensions for solutions that satisfy certain antisymmetry conditions, and Savin in \cite{sav} established its validity   for $4 \le n \le 8$ under the following additional natural hypothesis on the solution,
 \begin{equation}\label{asymp}
\lim_{x_n\to\pm\infty } u({x}',x_n)\to \pm 1.
\end{equation}
In \cite{wang}, Wang provided a new proof for the proof of Savin that involves more variational methods in the spirit.     Unlike the above proofs in dimensions $n\leq 5$, the proof of Savin is non-variational and does not use a Liouville type theorem. We refer interested readers to \cite{bar,bbg,bcn,bhm} for more information regrading this Liouville theorem for local semilinear equations.     For the case of $n\ge 9$,  del Pino, Kowalczyk and Wei in \cite{dkw} provided an example that shows the eight dimensions in the critical dimension for the conjecture. In addition, for the case of system of semilinear equations, Ghoussoub and the author in \cite{fg} established a counterpart of the above linear Liouville theorem for $H$-monotone solutions.  See \cite{mf} for a slightly improved version of this theorem as well.    Note that since  $u=(u_i)_{i=1}^m$ is a multi-component solution, the concept of monotonicity needs to be adjusted accordingly.   We borrow the following definition from \cite{fg}.    
    
 \begin{dfn}\label{Hmon} A solution $u=(u_k)_{k=1}^m$ of (\ref{main}) is said to be  $H$-monotone if the following hold,
\begin{enumerate}
 \item[(i)] For every $1\le i \le m$, each $u_i$ is strictly monotone in the $x_n$-variable (i.e., $\partial_{x_n} u_i\neq 0$).

\item[(ii)]  For all $i< j$,  we have 
  \begin{equation}\label{huiuj}
\hbox{$\partial_j H_i(u) \partial_n u_i(x) \partial_n u_j (x) > 0$  for all $x\in\mathbb {R}^n$.}
\end{equation}
\end{enumerate}
\end{dfn}
Note also that (ii) implies a combinatorial assumption on the sign of partial derivatives of $H_i$ and we call any system that admits such an assumption as {\it orientable} system.   In other words, for orientable systems there exists a sequence of sign functions $\theta=(\theta_i)_{i=1}^m$ where  $\theta_i\in \{-1,1\}$ such that $\partial_j H_i(u) \theta_i \theta_j >0$.   Counterparts of the above Liouville theorem for nonlocal equations with the fractional Laplacian operator are provided  by Cabr\'{e} and  Sol\'{a}-Morales in \cite{csm} and Cabr\'{e}  and Sire in \cite{cs1,cs2}.   In addition, for the case of system of equations  Sire and the author in \cite{fs} and the author in \cite{mf2} provided such a Liouville theorem.  We refer interested reader to \cite{ccinti} by Cabr\'{e} and Cinti for the fractional Laplacian case as well.   

This paper, in the light of the above results,  deals with one-dimensional symmetry results for bounded stable solutions of  (\ref{main}) via proving a linear Liouville theorem for the quotient of partial derivatives.  Note that for the case of fractional Laplacian, see  \cite{cs1,cs2,csm,fs,ccinti} and references therein,  the extension problem that is a local problem, derived by Caffarelli and Silvestre in \cite{cas},   plays a key role to prove one-dimensional symmetry results.  But, such an extension problem is not known for a general operator $L$ given in (\ref{Lui}).  Therefore, we apply the integral definition of the operator $L$ directly in our proofs.  Most of the ideas and computational techniques applied in this work, in regards to the operator $L$, are strongly motivated by the ones applied in \cite{ros,hv}.   Here is the notion of stability.  
 \begin{dfn} \label{stable}
A solution $u=(u_k)_{k=1}^m$ of (\ref{main}) is called stable when there exist a sequence of  functions   $\phi=(\phi_k)_{k=1}^m$ and a constant $\lambda \ge 0$  such that each $\phi_i$ does not change sign. In addition,  $\phi$ satisfies   the following 
 \begin{equation} \label{L}
L (\phi_i)= \sum_{j=1}^m \partial_j H_i(u)  \phi_j + \lambda \phi_i    \ \ \ \text{in}\ \ \mathbb R^n, 
  \end{equation} 
  where   $\partial_j H_i(u)  \phi_j \phi_i >0 $ for all $i<j$ when $1\le i,j\le m$. 
\end{dfn} 
It is straightforward to see that any $H$-monotone solution is a stable solution via differentiating (\ref{main}) with respect to $x_n$ and defining $\phi_i=\partial_{x_n} u_i$. As it is shown in \cite{abg},  in the absence of $H$-monotonicity and stability there are in fact two dimensional solutions for  Allen-Cahn systems in two dimensions. See also \cite{ali} by Alikakos for a comprehensive note regarding semilinear local systems.  Therefore, the concepts of $H$-monotonicity and stability seem to be crucial in the context.     Let us fix the following notation. 

\begin{note}\label{note1}
Throughout this paper we refer to sequences of functions $\psi=(\psi_i)_{i=1}^m$ and $\sigma=(\sigma_i)_{i=1}^m$ for $\psi_i:=\nabla u_i\cdot \eta$ where $\eta(x)=\eta(x',0):\RR^{n-1}\to \RR $ and $\sigma_i:=\frac{\psi_i}{\phi_i}$. 
\end{note}

Note that there is a geometric approach regrading the De Giorgi's conjecture in two dimensions, but in this paper our technique is applying Liouville theorems.    The  next definition is the notion of symmetric systems, introduced in \cite{mf}. The concept of symmetric systems seems to be crucial for providing De Giorgi type results for system (\ref{main}) with a general nonlinearity. 

\begin{dfn}\label{symmetric} We call system (\ref{main}) symmetric if the matrix of  partial derivatives of all components of $H=(H_i)_{i=1}^m$ that is $$\mathbb{H}:=(\partial_i H_j(u))_{i,j=1}^{m}$$  is symmetric. 
\end{dfn}

Occasionally,  we assume that the kernel $K$ satisfies the following conditions as this is the case in \cite{ros}.  
\begin{enumerate}
\item[(A)] The kernel $K(y)$ is nonnegative,  even and measurable function. In addition, $K$ has compact support in $B_1$  satisfying $K\equiv 0$ in $\RR^n\setminus B_1$ and $\int_{B_1} |y|^2 K(y) dy\le C$.

\item[(B)]  Harnack inequality: Suppose that $f$ does not change sign and is a solution of $L(f) +g(x) f =0 $  in $\RR^n$ for some $g\in L^\infty(\RR^n)$ then 
$\sup_{B_1(x_0)} f \le C\inf_{B_1(x_0)} f  $ for any $x_0\in\RR^n$. 

\item[(C)] H\"{o}lder estimate: suppose that  $f$ is a bounded solution of $L(f)=g$ in $\RR^n$ where $g$ is bounded in $B_1$. Then $||f||_{C^\alpha(B_{1/2})}\le C (  ||g||_{L^\infty{B_1}} +  ||f||_{L^\infty{\RR^n}} )$,
for some positive constants $\alpha,C$.   
\end{enumerate}

Note that the classical De Giorgi's conjecture with the stability assumption, instead of monotonicity, is known as the {\it stability conjecture} and it is known to be true for $n=2$ and not to be true in dimension $n=8$. For two dimensions the same proof for monotone solutions work for stability conjecture and for the eight dimensions we refer to \cite{pw} by Pacard and Wei.  This conjecture remains open for other dimensions. As a matter of fact our results in this work answer the stability conjecture in two dimensions for system of equations.

\section{Main results and proofs}
We start this section by deriving an integral system of equations for the sequence of functions $\sigma=(\sigma_i)_{i=1}^m$ where each $\sigma_i$ is the quotient of partial derivatives, see Notation \ref{note1}. 
\begin{lemma}\label{linear}
Suppose that $u=(u_i)_{i=1}^m$ is a stable solution of (\ref{main}). Then for each $i=1,\cdots,m$ and for any $x\in\RR^n$ we have
\begin{equation}\label{lineareq}
\int_{\RR^n} [\sigma_i(x)- \sigma_i(z)] \phi_i(z)  K(x-z) dz = \sum_{j=1}^m \partial_j H_i(u) \phi_j(x) [\sigma_j(x) -\sigma_i(x)].
\end{equation}
\end{lemma}
\begin{proof}
Since $u$ is a stable solution,   there exist a sequence of functions $\phi=(\phi_i)_{i=1}^m$ and a constant $\lambda\ge 0$  such that 
\begin{equation}\label{phi}
L(\phi_i(x))= \sum_{j=1}^m \partial_j H_i(u) \phi_j   + \lambda \phi_i. 
\end{equation}
Differentiating (\ref{main}) we get 
\begin{equation}\label{psi}
L(\psi_i(x))= \sum_{j=1}^m \partial_j H_i(u) \psi_j + \lambda \psi_i. 
\end{equation}
Since $\psi_i=\sigma_i\phi_i$, from (\ref{psi}) we have 
\begin{equation}\label{sigmaphi}
L(\sigma_i(x)(\phi_i(x))= \sum_{j=1}^m \partial_j H_i(u) \sigma_j(x)\phi_j(x). 
\end{equation}
Multiply (\ref{phi}) with $-\sigma_i$ and add with (\ref{sigmaphi}) to get 
\begin{equation}\label{LL}
L(\sigma_i(x)(\phi_i(x))- \sigma_i L(\phi_i(x))= \sum_{j=1}^m \partial_j H_i(u) \phi_j(x)(\sigma_j(x)-\sigma_i(x)). 
\end{equation}
From the definition of the operator $L$ we get 
\begin{eqnarray*}\label{Lphisigma}
L(\sigma_i(x)(\phi_i(x))- \sigma_i L(\phi_i(x)) &=& \phi_i(x) L(\sigma_i) - \int_{\RR^2} [\sigma_i(x)-\sigma_i(z)] [\phi_i(x)-\phi_i(z)] K(x-z) dz
\\ &=& \phi_i(x) \int_{\RR^n} [\sigma_i(x)-\sigma_i(z)] K(x-z) dz 
\\&&- \int_{\RR^2} [\sigma_i(x)-\sigma_i(z)] [\phi_i(x)-\phi_i(z)] K(x-z) dz
\\&=& \int_{\RR^2}  [\sigma_i(x)-\sigma_i(z)] \phi_i(z) K(x-z) dz.
\end{eqnarray*}  
This and (\ref{LL}) gives the desired result.
 
\end{proof}

As an immediate consequence of Lemma \ref{linear} we provide the following integral estimate.  

\begin{cor}\label{linearsum}
Suppose that assumptions of Lemma \ref{linear} hold. Assume also that $\eta=(\eta_i)_{i=1}^m$ where  $\eta_i\in C_c^1(\mathbb R^n)$ for each $1\le i\le m$. Then, 
\begin{eqnarray}
&&\sum_{i=1}^m   \int_{\RR^n} \int_{\RR^n} \sigma_i(x) [\sigma_i(x)- \sigma_i(z)] \phi_i(z) \phi_i(x)  K(x-z) \eta_i^2(x)dz dx
\\&&= \sum_{i,j=1}^m  \int_{\RR^n} \partial_j H_i(u) \phi_i(x)  \phi_j(x) [\sigma_j(x) -\sigma_i(x)] \sigma_i(x) \eta_i^2 (x)dx.
\end{eqnarray}
\end{cor}
\begin{proof}
Multiply (\ref{lineareq}) with $\phi_i(x) \eta^2_i(x)$ and integrate with respect to $x$. Then take sum on the index $i$. This finishes the proof.  

\end{proof}

Here is our main result. 

  \begin{thm} Suppose that $u=(u_i)_{i=1}^m$ is a bounded stable solution of symmetric system (\ref{main}) in two dimensions where the kernel $K$ satisfies conditions (A)-(C).   Then,  each $u_i$ must be a one-dimensional function for $i=1,\cdots,m$.  
  \end{thm}
  \begin{proof}From Corollary \ref{linearsum} we have
  \begin{eqnarray*}
&&\sum_{i,j=1}^m  \int_{\RR^n} \partial_j H_i(u) \phi_i(x)  \phi_j(x) [\sigma_j(x) -\sigma_i(x)] \sigma_i(x) \eta_i^2 (x)dx
\\&=&\sum_{i=1}^m   \int_{\RR^n} \int_{\RR^n} \sigma_i(x) [\sigma_i(x)- \sigma_i(z)] \phi_i(z) \phi_i(x)  K(x-z) \eta_i^2(x)dz dx
\\&=& -\sum_{i=1}^m   \int_{\RR^n} \int_{\RR^n} \sigma_i(z) [\sigma_i(x)- \sigma_i(z)] \phi_i(x) \phi_i(z)  K(x-z) \eta_i^2(z)dz dx .
\end{eqnarray*}
This implies that 
\begin{eqnarray*}
I&=& \sum_{i=1}^m   \int_{\RR^n} \int_{\RR^n}  [\eta_i^2(x) \sigma_i(x) -  \eta_i^2(z) \sigma_i(z)]    [\sigma_i(x)- \sigma_i(z)] \phi_i(x) \phi_i(z)  K(x-z)   dz dx
\\&=&2\sum_{i,j=1}^m  \int_{\RR^n} \partial_j H_i(u) \phi_i(x)  \phi_j(x) [\sigma_j(x) -\sigma_i(x)] \sigma_i(x) \eta_i^2 (x)dx .
\end{eqnarray*}
Set $\eta_i=\zeta_R$ where $\zeta_R\in C_c^\infty(\mathbb \RR^n)$ is the standard test function that is $\zeta_R=1$ in $B_R$ and $\zeta_R=0$ in $\RR^n\setminus B_{2R}$ with $||\nabla \zeta_R||_{L^{\infty}(B_{2R}\setminus B_R)}\le C R^{-1}$.  Note that for stable symmetric systems we have
\begin{eqnarray*}
 \sum_{i,j} \phi_i \phi_j \partial_j H_i(u) \sigma_i (\sigma_j-\sigma_i)&=& \sum_{i< j}   \phi_i \phi_j \partial_j H_i(u)  \sigma_i (\sigma_j-\sigma_i) + \sum_{i> j}  \phi_i \phi_j \partial_j H_i(u)  \sigma_i  (\sigma_j-\sigma_i)  \\&=&\sum_{i< j}  \phi_i \phi_j \partial_j H_i(u)   \sigma_i (\sigma_j-\sigma_i)  + \sum_{i< j}  \phi_i \phi_j \partial_j H_i(u)  \sigma_j (\sigma_i-\sigma_j)  \\&=&-\sum_{i< j}  \phi_i \phi_j \partial_j H_i(u)  (\sigma_j-\sigma_i)^2 \le 0. 
  \end{eqnarray*}
Therefore, $I\le 0$. Note that 
\begin{eqnarray}
2[\eta_i^2(x) \sigma_i(x) -  \eta_i^2(z) \sigma_i(z)] &=& [\sigma_i(x) -  \sigma_i(z)][\eta_i^2(x) + \eta_i^2(z)] 
\\&& + [\sigma_i(x) +  \sigma_i(z)][\eta_i^2(x) - \eta_i^2(z)] .
\end{eqnarray}
From this and the fact that $I\le 0$,  we get 
\begin{eqnarray*}
J(R)&=& \sum_{i=1}^m   \int_{\RR^n} \int_{\RR^n}  [\sigma_i(x) -  \sigma_i(z)]^2 [\zeta_R^2(x) + \zeta_R^2(z)] \phi_i(x) \phi_i(z)  K(x-z)   dz dx 
\\&\le &    \sum_{i=1}^m   \int_{\RR^n} \int_{\RR^n}  [\sigma^2_i(z) -  \sigma^2_i(x)] [\zeta_R^2(x) - \zeta_R^2(z)]  \phi_i(x) \phi_i(z)  K(x-z)   dz dx 
\\&=& \sum_{i=1}^m   \int_{\RR^n} \int_{\RR^n}  [\sigma_i(z) -  \sigma_i(x)] [\sigma_i(z) +  \sigma_i(x)]\\&&[\zeta_R(x) - \zeta_R(z)] [\zeta_R(x) + \zeta_R(z)]  \phi_i(x) \phi_i(z)  K(x-z)   dz dx . 
\end{eqnarray*}
Note that $J(R)\ge 0$.  From Cauchy-Schwarz inequality, we have
\begin{eqnarray*}
J^2(R)& \le & \left(\sum_{i=1}^m   \int_{\RR^n} \int_{\RR^n} [\sigma_i(x) -  \sigma_i(z)]^2 [\zeta_R^2(x) + \zeta_R^2(z)] \phi_i(x) \phi_i(z)  K(x-z)   dz dx \right)
\\&& \left(\sum_{i=1}^m   \int_{\RR^n} \int_{\RR^n} [\sigma_i(x) +  \sigma_i(z)]^2 [\zeta_R(x) - \zeta_R(z)]^2 \phi_i(x) \phi_i(z)  K(x-z)   dz dx \right) .
\end{eqnarray*}
From the definition of $\zeta_R$ it is straightforward to see that $$ (\zeta_R(x)-\zeta_R(z))^2 \le C R^{-2} |x-z|^2 \chi_{B_{3R}}(x) \ \ \text{when} \ |x-z|\le 1   .    $$
Note also that from Condition (A) we have $K\equiv 0$ on $\RR^2\setminus B_1$. Therefore,
\begin{eqnarray}\label{ll}
&& \sum_{i=1}^m   \int_{\RR^n} \int_{\RR^n} [\sigma_i(x) +  \sigma_i(z)]^2 [\zeta_R(x) - \zeta_R(z)]^2 \phi_i(x) \phi_i(z)  K(x-z)   dz dx 
\\&\le & \label{mm}   
R^{-2}  \sum_{i=1}^m  \int_{B_{2R}}  \int_{B_1} [\sigma_i(x) +  \sigma_i(z)]^2 \phi_i(x) \phi_i(z)   |x-z|^2 K(x-z)  dx dz . 
\end{eqnarray}
Applying  the condition (C) to each equation in (\ref{main}) we get $|\nabla u_i|\in L^\infty(\mathbb R^n)$ for each $i$.     From the definition of $\sigma_i$ then we have $|\sigma_i| \le ||\nabla u_i||_{ L^\infty(\mathbb R^n)} |\phi_i|^{-1}$ pointwise in $\RR^n$. Therefore, for each $i=1,\cdots,m$, 
$$[\sigma_i(x) +  \sigma_i(z)]^2 \le ||\nabla u_i||_{ L^\infty(\mathbb R^n)} \left( |\phi_i(x)|^{-2}  +  |\phi_i(z)|^{-2}   \right) .$$
This implies that the function in the integrand of (\ref{mm}) is bounded by 
$$ 
[\sigma_i(x) +  \sigma_i(z)]^2 \phi_i(x) \phi_i(z) \le ||\nabla u_i||_{ L^\infty(\mathbb R^n)}\left ( \frac{\phi_i(x)}{\phi_i(z)}   +  \frac{\phi_i(z)}{\phi_i(x)}  \right)  , 
$$
where the latter in bounded due to the condition (C). This and (\ref{mm}) implies that 
\begin{eqnarray}\label{nn}
&&  R^{-2}  \sum_{i=1}^m  \int_{B_{2R}}  \int_{B_1} [\sigma_i(x) +  \sigma_i(z)]^2 \phi_i(x) \phi_i(z)   |x-z|^2 K(x-z)  dx dz 
\\&\le& 
C R^{-2}  \sum_{i=1}^m \int_{B_{2R}}  \int_{B_1} |x-z|^2 K(x-z)  dx dz 
\\&\le& C R^{-2} |B_{2R}| , 
\end{eqnarray}
where we have used the condition (A) in the last inequality.  Note that in two dimensions  $ R^{-2} |B_{2R}| $ is bounded uniformly in $R$. This implies that the  integral in (\ref{ll}) is bounded uniformly in $R$, i.e., 
$$\sum_{i=1}^m   \int_{\RR^n} \int_{\RR^n} [\sigma_i(x) +  \sigma_i(z)]^2 [\zeta_R(x) - \zeta_R(z)]^2  \phi_i(x) \phi_i(z)  K(x-z)   dz dx \le C , $$
where $C$ is independent from $R$.  This implies that $J(R)\le C$ for any $R$ and therefore 
$$ \sum_{i=1}^m   \int_{\RR^n} \int_{\RR^n} [\sigma_i(x) -  \sigma_i(z)]^2  \phi_i(x) \phi_i(z)  K(x-z)   dz dx \le C.$$
The above upper integral bound on $J$ shows that 
$$ \sum_{i=1}^m   \int_{\RR^n} \int_{\RR^n}  [\sigma_i(x) -  \sigma_i(z)]^2 \phi_i(x) \phi_i(z)  K(x-z)   dz dx =0   . $$  
The fact that each $\phi_i$ does not change sign guarantees that  each $\sigma_i$ must be constant. From the definition of $\sigma_i$ we get $\nabla u_i\cdot \eta=C_i \phi_i$ in $\RR^2$. Therefore, each $u_i$ is a one-dimensional function. 

  \end{proof}

For the rest of this section, we analyze the stability of solutions.   For the  semilinear equation $-\Delta u=f(u)$ in $\RR^n$,  it is known that the pointwise stability that refers to the linearization of the operator implies a stability inequality of the form 
\begin{equation}\label{stabineq1}
 \int_{\RR^n} f'(u) \zeta^2 \le \int_{\RR^n} |\nabla \zeta|^2,
 \end{equation} 
for any test function $\zeta\in C_c^1(\RR^n)$. This inequality has been extensively used in the literature to study qualitative behaviour of solutions of the above equation.   As a matter of fact, it is shown by  Ghoussoub and Gui in \cite{gg1}, see also \cite{bcn} by Berestycki, Caffarelli and  Nirenberg,  that the stability inequality also implies the pointwise stability. This equivalence property is known for the case of system of semilinear equation by Ghoussoub and the author in \cite{fg} as Lemma 1.  In addition,  for the case of integral equations of the form (\ref{main}), when $m=1$, such equivalence is also known by Ros-Oton and Sire in \cite{ros}.  In the next lemmata we provide a similar result for system (\ref{main}). 

Here is the stability inequality that is a counterpart of (\ref{stabineq1}). 

\begin{lemma}\label{stabineq}  
  Let $u=(u_i)_{i=1}^m$ denote a stable solution of symmetric system (\ref{main}).  Then 
\begin{equation} \label{stability}
\sum_{i,j=1}^{m} \int_{\RR^n}  \partial_j H_i(u) \zeta_i(x) \zeta_j(x) dx \le  \frac{1}{2} \sum_{i=1}^{m} \int_{\RR^n} \int_{\RR^n}  [\zeta_i(x)- \zeta_i(z)]^2 K(z-x) dz dx , 
\end{equation} 
for any $\zeta=(\zeta_i)_i^m$ where $ \zeta_i \in C_c^\infty(\mathbb R^n)$ for $1\le i\le m$. 
\end{lemma}  
\begin{proof}
Suppose that $u=(u_i)_{i=1}^m$ denote a stable solution of (\ref{main}). Then there exist a sequence of functions $\phi=(\phi_k)_{k=1}^m$ and a constant $\lambda \ge 0$ such that   $\partial_j H_i(u)  \phi_j \phi_i >0 $ and 
\begin{equation} \label{L}
L(\phi_i)= \sum_{j=1}^m \partial_j H_i(u)  \phi_j + \lambda \phi_i   \ \ \ \text{in}\ \ \mathbb R^n. 
  \end{equation} 
Multiply both sides with $\frac{\zeta^2_i}{\phi_i}$ where $\zeta=(\zeta_i)_{i=1}^m$ is a sequence of test functions. Therefore, 
\begin{equation} \label{L1}
\sum_{i=1}^m L (\phi_i) \frac{\zeta^2_i}{\phi_i}=\sum_{i,j=1}^m \partial_j H_i(u) \frac{\phi_j}{\phi_i} \zeta^2_i + \lambda \zeta_i^2    \ \ \ \text{in}\ \ \mathbb R^n.
  \end{equation} 
 Note that the left-hand side can be rewritten as  
\begin{eqnarray*}
  \sum_{i,j=1}^{m}   \partial_j H_i(u) \phi_j \frac{\zeta_i^2}{\phi_i}& =&   \sum_{i<j}^{m}   \partial_j H_i(u) \phi_j \frac{\zeta_i^2}{\phi_i} + \sum_{i>j}^{n}   \partial_j H_i(u) \phi_j \frac{\zeta_i^2}{\phi_i} + \sum_{i=1}^m \partial_i H_i(u) {\zeta_i^2}
  %\\&=& \sum_{i<j}^{m}   \partial_j H_i(u) \phi_j \frac{\zeta_i^2}{\phi_i} + \sum_{i<j}^{m}   \partial_i H_j(u) \phi_i \frac{\zeta_j^2}{\phi_j} +\sum_{i=1}^m  \partial_i H_i(u) {\phi_i^2}
  \\&=& \sum_{i<j}^{m}    \left( \partial_j H_i(u) \phi_j \frac{\zeta_i^2}{\phi_i}+ \partial_i H_j(u) \phi_i\frac{\zeta_j^2}{\phi_j} \right)+  \sum_{i=1}^m \partial_i H_i(u) {\zeta_i^2}
  \\&\ge & 2 \sum_{i<j}^{m}     \partial_j H_i(u)  \zeta_i \zeta_j+ \sum_{i=1}^m \partial_i H_i(u) {\zeta_i^2}
  \\&=&\sum_{i,j=1}^{m}   \partial_j H_i(u)  \zeta_i\zeta_j.
  \end{eqnarray*}
From this and (\ref{L1}) we get 
\begin{equation} \label{LL1}
\sum_{i,j=1}^{m}  \int_{\RR^n} \partial_j H_i(u)  \zeta_i\zeta_j \le  \sum_{i=1}^m  \int_{\RR^n} L (\phi_i) \frac{\zeta^2_i}{\phi_i} . 
  \end{equation} 
Note that for each $i$ we have
\begin{equation}\label{Lphi}
\int_{\RR^n} L (\phi_i) \frac{\zeta^2_i}{\phi_i}=\frac{1}{2} \int_{\RR^n} \int_{\RR^n} [\phi_i(x) - \phi_i(z)] \left[ \frac{\zeta^2_i(x)}{\phi_i(x)}-  \frac{\zeta^2_i(z)}{\phi_i(z)} \right] K(x-z) dx dz. 
\end{equation}
The rest of the proof deals with simplifying the integrand of the right-hand side. To do so, we closely follow the computational techniques provided in \cite{ros,hv}.  Note that 
\begin{eqnarray}\label{phiix}
&&[\phi_i(x) - \phi_i(z)] \left[ \frac{\zeta^2_i(x)}{\phi_i(x)}-  \frac{\zeta^2_i(z)}{\phi_i(z)} \right]
\\&=& \nonumber 2[\phi_i(x) - \phi_i(z)] \left[ \zeta_i(x) -  \zeta_i(z)  \right] \left( \frac{ \zeta_i(x) + \zeta_i(z)  }{\phi_i(x) +\phi_i(z)}\right) \left(\frac{ [\phi_i(x) + \phi_i(z)]^2     }{4 \phi_i(x) \phi_i(z)}\right)
\\&& \nonumber - [\phi_i(x) - \phi_i(z)]^2 \left( \frac{ \zeta_i(x) + \zeta_i(z)  }{\phi_i(x) +\phi_i(z)}\right)^2 \left( \frac{2 \zeta_i^2(x)+2 \zeta_i^2(z)}{ [\zeta_i(x) + \zeta_i(z)   ]^2} \right) \left(\frac{ [\phi_i(x) + \phi_i(z)]^2     }{4 \phi_i(x) \phi_i(z)}\right) . 
\end{eqnarray}
The fact that for stable solutions, for each $i$ function $\phi_i$ does not change sign implies that  $\phi_i(x) \phi_i(z)=|\phi_i(x) \phi_i(z)|$.  On the other hand, we have
\begin{eqnarray*}
&&2[\phi_i(x) - \phi_i(z)] [ \zeta_i(x) -  \zeta_i(z) ] \left( \frac{ \zeta_i(x) + \zeta_i(z)  }{\phi_i(x) +\phi_i(z)}\right) 
\\&\le&  [ \zeta_i(x) -  \zeta_i(z) ]^2 + [\phi_i(x) - \phi_i(z)]^2 \left( \frac{ \zeta_i(x) + \zeta_i(z)  }{\phi_i(x) +\phi_i(z)}\right)^2  . 
\end{eqnarray*}
From this, (\ref{Lphi}) and (\ref{phiix}) we get 
$$\int_{\RR^n} L (\phi_i) \frac{\zeta^2_i}{\phi_i} \le  \frac{1}{2} \sum_{i=1}^{m} \int_{\RR^n} \int_{\RR^n}  [\zeta_i(x)- \zeta_i(z)]^2 K(z-x) dz dx.$$
This and (\ref{LL1}) finishes the proof.

\end{proof}

We now provide the other side of the argument to conclude the equivalence of the linearized stability (\ref{L}) and the stability inequality (\ref{stability}).   Note that the following proof is strongly motivated by the one given by Ghoussoub and the author in \cite{fg} for the case of semilinear local systems. 

\begin{lemma}\label{equiv}
Suppose that  $u=(u_i)_{i=1}^m$ is a solution of symmetric system (\ref{main}) for which the inequality (\ref{stability}) holds. Assume that conditions (B) and (C) hold as well. Then $u$ is a stable solution of (\ref{main}).  
\end{lemma}
\begin{proof}
For each $R>0$ define  
  \begin{equation}
  \lambda_1(R) := \inf  \left\{ T_R(\zeta) ;  \zeta=(\zeta_i)_{i=1}^{m} \ \text{where} \  \zeta_i \in H_K (\RR^n) ,  \zeta_i=0 \ \text{in}\  \RR^n\setminus B_R,  \ \sum_{i=1}^m \int_{B_R} \zeta_i^2=1  \right\}, 
  \end{equation}
  where  $H_K(B_R)$ is the closure of $C_c^\infty(\RR^n)$ with the norm 
  $$||f||_{H_K}(\RR^n):= \int_{\RR^n} \int_{\RR^n} |f(x)- f(z) |^2 K(x-z) dxdz, $$  
and the operator $T_R$ is given by 
  \begin{eqnarray}
T_R(\zeta) &:=&  \sum_{i,j=1}^{m} \int_{\RR^n}  \partial_j H_i(u)  \zeta_i(x) \zeta_j(x) dx \\&& - \frac{1}{2} \sum_{i=1}^{m} \int_{\RR^n} \int_{\RR^n}  [\zeta_i(x)- \zeta_i(z)]^2 K(z-x) dz dx .
  \end{eqnarray}
From (\ref{stability}) we have $\lambda_1(R)\ge 0$ and there exist eigenfunctions $\zeta_i^R$ such that 
 for each $i=1,\cdots,m$, 
  \begin{equation}\label{zetar}
  \left\{
                      \begin{array}{ll}
                                           L(\zeta_i^R)= \sum_{j=1}^m \partial_j H_i(u)  \zeta^R_j + \lambda_1(R) \zeta_i^R, & \hbox{if $|x|< R$,} \\
                       \zeta_i^R=0, & \hbox{if $|x|= R$.}
                                                                       \end{array}
                    \right.\end{equation}
 Since the system  is orientable, there exists $(\theta_i)_{i=1}^{m}$ such that $ \partial_j H_i(u) \theta_i\theta_j >0.$ We now apply the maximum principle and use the sign of each $\theta_i$ to assign signs for the eigenfunctions $(\zeta_k^R)_k$ so that they satisfy 
\begin{equation}\label{zetamono}
 \partial_j H_i(u)  \zeta_i^R\zeta_j^R> 0. 
 \end{equation} 
Here $\zeta_i^R$ can be replaced by $sgn(\theta_i)|\zeta_i^R|$ if needed and we can normalize each $\zeta_i$ so that 
  \begin{equation}\label{normal}
  \sum_{i=1}^m  |\zeta_k^R(0)|=1.
  \end{equation}
  Note that $\lambda_1(R) \downarrow \lambda_1 \ge 0$ as $R \to \infty$.  Define $\chi_i^R:=sgn(\zeta_i^R) \zeta_i^R$ and multiply system (\ref{zetar}) with $sgn(\zeta^R_i)$ to get that 
 \begin{equation}\label{chir}
  \left\{
                      \begin{array}{ll}
                                         L (\chi_i^R) = \partial_i H_{i}(u) \chi_i^R +  \sum_{j\neq i} |\partial_j H_i(u)|  \chi_j^R+   \lambda_1(R)  \chi_i^R , & \hbox{if $|x|< R$,} \\
                       \chi_i^R=0, & \hbox{if $|x|= R$.}
                                                                       \end{array}
                    \right.\end{equation}
 Here we have used the fact that from orientability,  $sgn{(\zeta_i^R \zeta_j^R) }=sgn(\partial_j H_i(u)) $ when $i\neq j$.  Each $\chi_i^R$ is a  nonnegative function and $\chi=(\chi_i)_{i=1}^m$ is a solution for (\ref{chir}). From the Harnack's inequality, given in the condition (C), for any compact subset $\Omega$, $\max_\Omega |\chi_i^R|	\le C(\Omega) \min_\Omega |\chi_i^R| $ for all $i=1,\cdots,m$ with the latter constant being independent of $\chi_i^R$. From elliptic estimates, given as the condition (B), one can see that the family $(\chi_i^R)_R$ have 
uniformly bounded derivatives on compact sets. So, for a subsequence $(R_k)_k$ going to infinity, $(\chi_i^{R_k})_k$  converges in $C^2_{loc} (\mathbb{R}^n)$ to some $\chi_i \in C^2(\mathbb {R}^n)$ and that $\chi_i\ge0$. From (\ref{chir}) we see that $\chi_i$ satisfies 
\begin{eqnarray}\label{chi}
 L (\chi_i) &=& \partial_i H_{i}(u) \chi_i +  \sum_{j\neq i} |\partial_j H_i(u)|  \chi_j+   \lambda_1  \chi_i \\&\ge&\nonumber  (\partial_i H_{i}(u) + \lambda_1 )\chi_i 
  \end{eqnarray} 
Note that  $\chi_i\ge 0$ and $\partial_i H_{i}(u)$ is bounded. The strong maximum principle  implies that  either $\chi_i=0$ or $\chi_i>0$ in $\mathbb{R}^n$.  Suppose that $\chi_i=0$, then  from (\ref{chi}) we have $ \sum_{j\neq i} |\partial_j H_i(u)|  \chi_j=0$ that is $\chi_j=0$ if $j\neq i$. This  is in contradiction with (\ref{normal}).  Therefore,   $\chi_i>0$ for all $i=1,\cdots,m$.    Finally, set  $\phi_i:=sgn(\theta_i) \chi_i$ for $i=1,\cdots, m$ and observe that  $(\phi_i)_i$ satisfies (\ref{L}) and that $\partial_j H_i(u) \phi_j \phi_i > 0$ for $i<j$, which means that $u=(u_i)_{i=1}^m$ is a  stable solution.  

\end{proof}

{\it Acknowledgment}.    The author would like to thank Yannick Sire for his comments and online discussions. We thank ETH Zurich for the hospitality where parts of this work were  completed.

\end{document}